\documentclass[11pt]{article}

\usepackage{amsmath, amsfonts, amsthm, verbatim}

\newcommand{\R}{\mathbb{R}}
\newcommand{\N}{\mathbb{N}}
\newcommand{\Z}{\mathbb{Z}}
\newcommand{\E}{\mathbf{E}}
\newcommand{\p}{\mathbf{P}}

\theoremstyle{plain}
\newtheorem{lemma}{Lemma}
\newtheorem{theorem}{Theorem}
\newtheorem*{theorem*}{Theorem}
\newtheorem{proposition}{Proposition}

\theoremstyle{remark}
\newtheorem{remark}{Remark}
\newtheorem*{example}{Example}

\title{Implicit renewal theory in the arithmetic case}

\author{P\'eter Kevei
\\
Center for Mathematical Sciences, Technische Universit\"at M\"unchen \\
Boltzmannstra{\ss }e 3, 85748 Garching, Germany \\
\texttt{peter.kevei@tum.de}}

\date{}

\begin{document}

\maketitle

\begin{abstract}
We extend Goldie's implicit renewal theorem to the arithmetic case, which allows us to 
determine the tail behavior of the solution of various random fixed point equations. It 
turns out that the arithmetic and nonarithmetic cases are very different. Under 
appropriate conditions we obtain that the tail of the solution $X$ of the fixed 
point equations $X \stackrel{\mathcal{D}}{=} AX + B$, $X \stackrel{\mathcal{D}}{=} AX 
\vee B$  is $\ell (x) q(x) x^{-\kappa}$, where $q$ is a logarithmically periodic function 
$q(x e^h) = q(x)$, $x > 0$, with $h$ being the span of the arithmetic distribution 
of $\log A$, and $\ell$ is a slowly varying function. In particular, the tail is not 
necessarily regularly varying.
We use the renewal theoretic approach developed by Grincevi\v{c}ius and Goldie.

\end{abstract}

\section{Introduction}

Consider the perpetuity equation
\begin{equation} \label{eq:perpetuity}
X \stackrel{\mathcal{D}}{=} A X + B,
\end{equation}
where $(A,B)$ and $X$ on the right-hand side are independent. The tail behavior of the 
solution has attracted much attention since Kesten's result \cite{Kesten}. This result 
was rediscovered by Grincevi\v{c}ius \cite{Grinc}, whose renewal theoretic method was 
developed further and applied to more general random fixed point equations by Goldie 
\cite{Goldie}. They proved the following.

\begin{theorem*}\emph{(Kesten--Grincevi\v{c}ius--Goldie)}
Assume that $A \geq 0$ a.s., $\E A^\kappa = 1$ for some $\kappa > 0$, $\E A^\kappa \log_+ 
A < \infty$, $\E |B|^\kappa < \infty$ and the distribution of $\log A$ conditioned on $A 
\neq 0$ is nonarithmetic. Then
\[
\lim_{x \to \infty} x^\kappa \, \p \{ X > x \} = c_+, \quad 
\lim_{x \to \infty} x^\kappa \, \p \{ X < -x \} = c_-.
\]
Furthermore, if $\p \{ Ax + B = x \} < 1$ for all $x \in \R$, then $c_+ + c_- > 0$.
\end{theorem*}

Besides perpetuity equation (\ref{eq:perpetuity}) the best known and most 
investigated random fixed point equation is the maximum equation
\begin{equation} \label{eq:maxeq}
X \stackrel{\mathcal{D}}{=} AX \vee B,
\end{equation}
where $a \vee b = \max \{ a,b \}$, $A \geq 0$ and $(A,B)$ and $X$ on the right-hand 
side are independent. This equation appears in the analysis of the maximum 
of a perturbed random walk. Under the same assumptions Goldie proved the same tail 
behavior of the solution. For theory, applications and history of perpetuity equation 
(\ref{eq:perpetuity}) we refer to Buraczewski, Damek and Mikosch \cite{BDM} and for 
perturbed random walks and maximum equation (\ref{eq:maxeq}) to Iksanov \cite{Iksanov}.

Interestingly enough, the case when $\log A$ is arithmetic was only treated by 
Grincevi\v{c}ius in the perpetuity case and by Iksanov \cite{Iksanov} in the maximum 
case, see their theorems below. In both cases the tail has a completely different 
behavior than in the nonarithmetic case. In particular, the tail is not 
regularly varying. Investigating the maximum of random walks the maximum equation 
(\ref{eq:maxeq}) appears with $B \equiv 1$. In this case the tail behavior was 
analyzed by Asmussen \cite[XIII. Remark 5.4]{Asmussen} and by Korshunov 
\cite{Kor}.

The aim of the present paper is to extend Goldie's implicit renewal theorem to the 
arithmetic case, providing a unified approach for random fixed point equations. This is 
done in Subsection \ref{subsect:finite}. As an example we prove that
the St.~Petersburg distribution is a solution of an appropriate perpetuity equation, 
showing that the tail of a solution can be irregular. We also show that the set of 
possible functions appearing in the tail of the solution is large. In Subsection 
\ref{subsect:infinite} we treat the case when 
the condition $\E A^\kappa \log_+ A < \infty$ does not hold, while Subsection 
\ref{subsect:<1} deals with the case $\E A^\kappa < 1$, but $\E A^t = \infty$, for $t > 
\kappa$. The corresponding nonarithmetic versions were treated by Kevei \cite{Kevei}. In 
each case we give the general implicit renewal theorem and then specialize it to the two 
equations (\ref{eq:perpetuity}) and (\ref{eq:maxeq}). Finally, in Subsection 
\ref{subsect:sandwich} using Alsmeyer's sandwich technique \cite{Alsmeyer} we show how 
these results apply to iterated function systems. All the proofs are contained in Section 
\ref{sect:proofs}.

\section{Results and discussion}

A random variable $Y$, or its distribution, is called \emph{arithmetic} 
(also called centered arithmetic, or centered lattice) if
$Y \in h \Z = \{ 0, \pm h, \pm 2h, \ldots \}$ a.s.~for some $h > 0$. The 
largest such an $h$ is the span of $Y$. We stress the difference between arithmetic and 
lattice distributions, where the latter means $Y \in a + h \Z$ a.s.~for some $a, h$.

Assume that $\E A^\kappa = 1$ for some $\kappa > 0$, which is the so-called Cram\'er 
condition (for $\log A$). Due to the multiplicative structure in 
(\ref{eq:perpetuity}) and (\ref{eq:maxeq}), the key idea, which goes back to 
Grincevi\v{c}ius, is to introduce a new probability measure 
\begin{equation} \label{eq:def-pkappa}
\p_\kappa \{  \log A \in C \} = \E [ I(\log A \in C) A^\kappa],
\end{equation}
where $C$ is a Borel set of $\R$, and $I(B)$ is the indicator function of the event 
$B$, i.e.~it is 1 if $B$ holds, and 0 otherwise. Under the new measure the distribution 
function (df) of $\log A$ is
\begin{equation} \label{eq:Fkappa}
F_\kappa(x) = \p_\kappa \{ \log A \leq x \} =
\int_{-\infty}^x e^{\kappa y} F(\mathrm{d} y),
\end{equation}
where $F(x) = \p \{ \log A \leq x \}$. We use the convention 
$\int_a^b = \int_{(a,b]}$ for $-\infty < a < b < \infty$.
Note that without any further assumption on the distribution of $A$ we have
\begin{equation} \label{eq:Fknegtail}
F_\kappa(-x) \leq e^{-\kappa x} \quad \text{for } x > 0. 
\end{equation}
Under the new measure equations 
(\ref{eq:perpetuity}), (\ref{eq:maxeq}) can be rewritten as renewal equations, where the 
renewal function is 
\begin{equation} \label{eq:def-U}
U(x) = \sum_{n=0}^\infty F_\kappa^{*n}(x), 
\end{equation}
${*n}$ standing for the usual $n$-fold convolution.
Then the tail asymptotics can be obtained via the key renewal theorem in the arithmetic 
case on the whole line (note that $\log A$ can be negative). If $\E_\kappa \log A < 
\infty$, to which we refer as the `finite mean case', the required key renewal theorem is 
given in \cite[Proposition 6.2.6]{Iksanov}. In the `infinite mean case', when $\E_\kappa 
\log A = \infty$, but $F_\kappa$ has regularly varying tail we prove an infinite mean key 
renewal theorem in the arithmetic case in Lemma \ref{lemma:keyren}, which is an extension 
of Erickson's result \cite[Theorem 3]{Erickson}. Finally, when Cram\'er's condition does 
not hold, i.e.~$\E A^\kappa = \theta \in (0,1)$, $\E A^t = \infty$, $t > \kappa$, one 
ends up with a defective renewal equation, for which a key renewal theorem is given in 
Lemma \ref{lemma:keyren2}.

\subsection{Finite mean case} \label{subsect:finite}

Our assumptions on $A$ are the following:
\begin{equation} \label{eq:A-ass1}
\begin{gathered}
A \geq 0, \ \E A^\kappa = 1 \text{ for some $\kappa > 0$},
\  \E A^\kappa \log_+ A < \infty, \\
\text{and  $\log A$ conditioned on $A \neq 0$ is arithmetic with span $h$}.
\end{gathered}
\end{equation}

Note that the convexity of the function 
$\E A^s$, $s \in [0,\kappa]$ and  $\E A^0 = \E A^\kappa = 1$ together implies
$\E_\kappa \log A = \E A^\kappa \log A =: \mu > 0$. Moreover, 
(\ref{eq:Fknegtail}) implies that $\E_\kappa [(\log A)_{-}]^2 < \infty$. Therefore the 
renewal function $U$ in (\ref{eq:def-U}) is well-defined, see Theorem 2.1 by Kesten and 
Maller \cite{KestenMaller}.

For a real function $f$ the set of its continuity points is denoted by $C_f$. 
Introduce the notation
\[
\begin{split}
\mathcal{Q} = \Big\{  & q: (0,\infty) \to [0,\infty) \, : \,
x^{-\kappa} q(x) \text{ is nonincreasing for some } \kappa > 0, \\
&\ q(x e^h) = q(x), \ \forall x >0, \text{ for some } h > 0 \Big\}.
\end{split}
\]
In all the statements below a function $q \in \mathcal{Q}$ appears in the tail 
asymptotics. Note that $q \in \mathcal{Q}$ is either strictly positive or identically 0. 
The following result is a counterpart of Goldie's implicit renewal theorem \cite[Theorem 
2.3]{Goldie} in 
the arithmetic case.

\begin{theorem} \label{th:imp}
Assume (\ref{eq:A-ass1}) and for a random variable $X$
\begin{equation} \label{eq:int-cond}
\int_0^\infty y^{\kappa -1 } | \p \{ X > y \} - \p \{ AX > y \} | \mathrm{d} y < \infty, 
\end{equation}
where $A$ and $X$ are independent.
Then there exists a function $q \in \mathcal{Q}$ such that for $x \in C_q$
\begin{equation} \label{eq:imp-lim}
\lim_{n \to \infty} x^\kappa e^{\kappa nh}  \p \{ X > x e^{ nh} \} = q(x).
\end{equation}
Moreover, if
\begin{equation} \label{eq:sum-cond}
\sum_{j \in \Z} e^{\kappa(x + j h)} \big| \p \{ X > e^{x + j h} \}
- \p \{ AX > e^{x + j h} \}  \big| < \infty, \quad \text{for each } x \in \R,
\end{equation}
then (\ref{eq:imp-lim}) holds for all $x > 0$.
\end{theorem}

Whenever $q$ is continuous the exact tail asymptotic can be determined. In Proposition 
\ref{prop:Qset} below, we show that $q$ indeed can be continuous.

\begin{lemma} \label{lemma:cont-q}
If the function $q \in \mathcal{Q}$ in (\ref{eq:imp-lim}) is nonzero and continuous, then
\begin{equation} \label{eq:cont-q}
\p \{ X > x \} \sim \frac{q(x)}{x^{\kappa}} \quad \text{as } \, x \to \infty. 
\end{equation}
\end{lemma}

Similar behavior appears in the theory of semistable and max-semistable 
laws, and in the theory of smoothing transformaiton. If $\kappa \in (0,2)$ then $q(x) 
x^{-\kappa}$ is exactly the tail of the L\'evy 
measure of a semistable law with $q \in \mathcal{Q}$. For 
$\kappa > 0$ the function $\exp \{ - q(x) x^{-\kappa} \}$, $x > 0$, is a max semistable 
distribution function. For more in this direction we refer to Meerschaert and Scheffler 
\cite{MeerschaertS} and Megyesi \cite{Megyesi, Megyesi2}.

The smoothing transformation is closely related to our setup. Consider the fixed 
point equation
\begin{equation} \label{eq:smoothing}
X \stackrel{\mathcal{D}}{=} A_1 X_1 + \ldots + A_N X_N, 
\end{equation}
where $X_1, \ldots, X_N$ are iid copies of $X$,  $A_1, \ldots, A_N$ are iid, and the 
$A$'s and $X$'s are independent. Durrett and Liggett \cite{DL} gave necessary and 
sufficient conditions for the existence of the solution of (\ref{eq:smoothing}). Assuming 
existence, let $\varphi$ be the Laplace-transform of the solution. 
In \cite[Theorem 2]{DL} it is shown that, under appropriate conditions, 
$1 - \varphi(t) \sim t^\alpha h(t)$ as $t \to 0$ for some $\alpha \in (0,1]$, where $h$ 
is a logarthmically periodic function. It is not clear how the tail behavior can be 
inferred from these results, since to apply Tauberian theorems regular variation is 
necessary. For more results and references see Alsmeyer, Biggins and Meiners \cite{ABM}, 
in particular Corollary 2.3 and Theorem 3.3.

Finally, we mention that functions of the form $f(x) = p(x) e^{\lambda x}$, $\lambda \in 
\R$, where $p$ is a periodic function, are the solutions of certain integrated Cauchy 
functional equations, see Lau and Rao \cite{LauRao}.

\smallskip

Consider the perpetuity equation (\ref{eq:perpetuity}). We present Grincevi\v{c}ius's 
result in the arithmetic case below. The slight improvement is the 
positivity of $q$, which follows from Goldie's argument \cite[p.~157]{Goldie} combined 
with Theorem 1.3.8 \cite{Iksanov} by Iksanov.

\begin{theorem}{\emph{(Grincevi\v{c}ius \cite[Theorem 2]{Grinc})}} \label{th:grinc}
Assume (\ref{eq:A-ass1}) and
$\E |B|^\kappa < \infty$. Let $X$ be the unique solution of (\ref{eq:perpetuity}). Then 
there exist functions $q_1, q_2 \in \mathcal{Q}$ such that 
\begin{equation} \label{eq:grinc-lim}
\begin{split}
& \lim_{n \to \infty} x^\kappa e^{\kappa nh}  \p \{ X > x e^{ nh} \} = q_1(x), 
\quad x \in C_{q_1},\\
& \lim_{n \to \infty} x^\kappa e^{\kappa nh}  \p \{ X < -x e^{ nh} \} = q_2(x), \quad x 
\in C_{q_2}.
\end{split}
\end{equation}
If $\p \{ A x + B = x\} < 1$ for any $x \in \R$, then $q_1(x) + q_2 (x) > 0$.
\end{theorem}

Grincevi\v{c}ius also showed that (\ref{eq:grinc-lim}) holds for all $x \in \R$ if $B 
\geq 0$ a.s.

The corresponding maximum equation was treated by Iksanov.

\begin{theorem}\emph{(Iksanov \cite[Theorem 1.3.8]{Iksanov})} \label{th:max}
Assume (\ref{eq:A-ass1}) and $\E |B|^\kappa < \infty$. Let $X$ be the unique solution of 
(\ref{eq:maxeq}). Then there exists a function $q \in \mathcal{Q}$ such that for any $x > 
0$
\begin{equation} \label{eq:max-lim}
\lim_{n \to \infty} x^\kappa e^{\kappa nh}  \p \{ X > x e^{nh} \} = q(x).
\end{equation}
If $B \geq 0$ a.s.~and $\p \{ B > 0 \} > 0$ then $q(x) > 0$.
\end{theorem}

In fact this theorem is stated under the additional condition $B > 0$ a.s. In 
the context of \cite{Iksanov} this condition automatically holds since $B = e^\eta$ for 
some random variable $\eta$.

Note the difference between the two theorems. In case of equation (\ref{eq:maxeq}) it is 
possible to show that the stronger condition 
(\ref{eq:sum-cond}) holds (see the proof of \cite[Theorem 1.3.8]{Iksanov}), while in the 
perpetuity case (\ref{eq:perpetuity}) one only has the weaker condition 
(\ref{eq:int-cond}).

The formula for the function $q(x)$ (given in the proof below) in Theorem 
\ref{th:imp} is complicated and implicit, 
since it contains the tail of the solution $X$. Therefore one might think that $q(x) 
\equiv c$ and the tail is simply $c x^{-\kappa}$ as in the nonarithmetic case. We first 
give an explicit example which shows that this is not the case, i.e.~the function $q$ 
can be 
nonconstant.

\begin{example}
The St.~Petersburg game is defined as follows. Peter tosses a fair coin until it lands 
heads and pays $2^k$ ducats to Paul if this happens at the $k$th toss. If $X$ denotes 
Paul's winning then $\p \{ X = 2^k \} = 2^{-k}$, $k = 1,2,\ldots$. The distribution 
function of $X$ is
\begin{equation*} 
\p \{ X \leq x \} =
\begin{cases}
1 - \frac{2^{\{ \log_2 x \}} }{x}, & x \geq 2, \\
0, & x < 2,
\end{cases}
\end{equation*}
where $\lfloor x \rfloor = \max \{ k \in \Z: \, k \leq x \}$ is the usual (lower) integer 
part of $x$,  $\lceil x \rceil = - \lfloor -x \rfloor$ stand for the upper integer part 
and $\{ x \} = x - \lfloor x \rfloor $ is the fractional part.
We note that this 
distribution does not belong to the domain of attraction of any stable law, since the 
function $2^{\{ \log_2 x \}}$ is not slowly varying at infinity. For further properties 
and history of the St.~Petersburg games we refer to \cite{Csorgo} and \cite{BGyK} and the 
references therein.

We show that $X$ is the solution of a perpetuity equation, where the joint distribution 
of $(A,B)$ in (\ref{eq:perpetuity}) is the following:
\begin{equation} \label{eq:StPAB}
\p \{ A = 0, B = 2^k \} = 2^{-2k}, \ k=1,2,\ldots, \quad
\p \{ A= 2^\ell, B= 0 \} = 2^{-(2\ell+1)}, \ \ell=0,1,\ldots. 
\end{equation}
Indeed, assume that $X$ is independent of $(A,B)$. Then for $k \geq 1$
\[
\begin{split}
\p \{ AX + B = 2^k \} & = 
\sum_{\ell=0}^{k-1} \p\{ A= 2^\ell, B= 0 \} \, \p \{ X = 2^{k - \ell} \}
+ \p \{ A= 0, B= 2^k \} \\
& = \sum_{\ell=0}^{k-1} 2^{-(2\ell+1)} \, 2^{-(k - \ell)} +  2^{-2k} \\
& = 2^{-k-1} \, 2 \, (1 - 2^{-k}) + 2^{-2k} = 2^{-k}.
\end{split}
\]
Moreover, $\log A$ conditioned on $A$ being nonzero is arithmetic with span
$h = \log 2$, and
\[
\E A = \sum_{k=0}^\infty \p\{ A = 2^k \} \, 2^k = 1, \ \E A \log_+ A < \infty, \
\E B < \infty.
\]
That is the conditions of Theorem \ref{th:grinc} are satisfied with $\kappa = 1$. In this 
special case we see that $ q(x) = 2^{\{ \log_2 x \} }$.
\end{example}

What simplifies the analysis of the perpetuity equation with $(A,B)$ in 
(\ref{eq:StPAB}) is that $AB = 0$ a.s. It is worth mentioning that whenever $AB=0$, 
$B\geq 0$ a.s.~the solutions of perpetuity equation (\ref{eq:perpetuity}) and maximum 
equation (\ref{eq:maxeq}) take the same form $X=A_1 \ldots A_{N-1} B_N$ for appropriate 
geometrically distributed $N$ (see the proof of Proposition \ref{prop:Qset} for more 
details).  In particular, the St.~Petersburg  distribution is the solution of 
(\ref{eq:maxeq}) with $(A,B)$ in (\ref{eq:StPAB}).
 
Now we generalize this example and show that the set of all possible $q$ functions in 
Theorems \ref{th:grinc} and \ref{th:max} contains the set of right-continuous nonzero 
functions in $\mathcal{Q}$.

\begin{proposition} \label{prop:Qset}
Let $q, q_1, q_2 \in \mathcal{Q}$ right-continuous functions such that $q \neq 0$ and 
$q_1 + q_2 \neq 0$. 
Then there exists $(A,B)$ satisfying the conditions of Theorem \ref{th:grinc} such that 
for the tail of the unique solution of (\ref{eq:perpetuity}) the asymptotic 
(\ref{eq:grinc-lim}) holds with the prescribed $q_1, q_2$.
Furthermore, there exists $(A,B)$ satisfying the conditions of Theorem \ref{th:max} 
such that for the tail of the unique solution of (\ref{eq:maxeq}) the asymptotic 
(\ref{eq:max-lim}) holds with the prescribed $q$.
\end{proposition}

In the proof of this statement we give an explicit construction of $(A,B)$. 
In fact, for $\kappa =1$, $h = \log 2$ the distribution of $A$ is (almost) the same as in 
the example above, and only the distribution of $B$ depends on $q$. When $q(x) \equiv q$ 
is constant, Lemma \ref{lemma:cont-q} implies that the tail of the solution $X$ is 
regularly varying, more precisely $\p \{ X > x \} \sim q x^{-\kappa}$ as $x \to \infty$. 
An explicit example is given in the proof of Proposition \ref{prop:Qset}.

However, for general $(A,B)$ it seems very difficult to determine $q$. It would be 
interesting to know what conditions on $(A,B)$ imply that $q$ is constant, or $q$ is 
continuous, but these questions do not seem to be tractable with our methods.

\subsection{Infinite mean case} \label{subsect:infinite}

Now we assume that $F_\kappa$ in (\ref{eq:Fkappa}) belongs to the domain of attraction 
of an $\alpha$-stable law with $\alpha \in (0,1]$, that is
\begin{equation} \label{eq:regvar}
1 - F_\kappa(x) =: \overline F_\kappa(x) = \frac{\ell(x)}{x^{\alpha}}, 
\end{equation}
where $\ell$ is a slowly varying function. 
Furthermore, we assumat that the mean is  infinite if $\alpha = 1$.
Introduce the truncated expectation
\begin{equation} \label{eq:def-m}
m(x) = \int_0^x \overline F_\kappa(y) \mathrm{d} y. 
\end{equation}
Simple properties of regularly varying functions imply 
$m(x) \sim \ell(x) x^{1-\alpha}/ (1-\alpha)$ for $\alpha \neq 1$, and $m$ is slowly 
varying for $\alpha = 1$. Recall $U$ from 
(\ref{eq:def-U}) and put $u_n = U( nh )- U(nh-)$. 
Note that $U(x) < \infty$ for all $x \in \R$, since the random walk
$(S_n = \log A_1 + \ldots + \log A_n)_{n \geq 1}$ drifts to infinity under $\p_\kappa$ 
and $\E_\kappa [(\log A)_-]^2 < \infty$ by (\ref{eq:Fknegtail}); see 
Theorem 2.1 by Kesten and Maller \cite{KestenMaller}.
In this case the Blackwell theorem only states that $u_n \to 0$. The so-called strong 
renewal theorem (SRT) gives the exact rate, namely
\begin{equation} \label{eq:SRT}
\lim_{n \to \infty} u_n m(nh) = h \, C_\alpha, \quad 
C_\alpha = \frac{\sin( \alpha \pi)}{(1-\alpha)\pi},
\end{equation}
with the convention $C_1 =1$. The first infinite mean SRT in the arithmetic case was 
shown by Garsia and Lamperti 
\cite{GL}, who proved that (\ref{eq:SRT}) holds for $\alpha \in (1/2,1)$, and under some 
extra assumptions, for $\alpha \leq 1/2$. Their results were extended to the 
nonarithmetic case by Erickson \cite{Erickson}, who also showed (\ref{eq:SRT}) for 
$\alpha =1$, see \cite[formula (2.4)]{Erickson}. Necessary and sufficient conditions for 
the SRT for nonnegative random variables were obtained by Caravenna \cite{Caravenna} and 
Doney \cite{Doney}. It was pointed out in \cite[Appendix]{Kevei} that their result 
extends to our case, where the random variable is not necessarily positive but the left 
tail is exponential. It turned out that if (\ref{eq:regvar}) holds with $\alpha \in 
(0,1/2]$ then (\ref{eq:SRT}) holds if and only if
\begin{equation} \label{eq:Doney-cond}
\lim_{\delta \to 0} \limsup_{x \to \infty} x \overline F_\kappa(x) 
\int_1^{\delta x} \frac{1}{y \overline F_\kappa(y)^2} F_\kappa(x - \mathrm{d} y) = 0.
\end{equation}
It is also shown in \cite{Caravenna, Doney} that for $\alpha > 1/2$ condition 
(\ref{eq:Doney-cond}) automatically holds.

Summarizing, our assumptions on $A$ are the following:
\begin{equation} \label{eq:A-ass2}
\begin{gathered}
A \geq 0, \ \E A^\kappa = 1, \ \text{(\ref{eq:regvar}) and (\ref{eq:Doney-cond}) holds 
for $F_\kappa$ for some $\kappa > 0$ and $\alpha \in (0,1]$,} \\
\text{and  $\log A$ conditioned on $A \neq 0$ is arithmetic with span $h$}.
\end{gathered}
\end{equation}

Recall the definition of $m$ from (\ref{eq:def-m}), and that $m$ is regularly varying.

\begin{theorem} \label{th:imp2}
Assume (\ref{eq:A-ass2}) and for a random variable $X$
\begin{equation} \label{eq:delta-int}
\int_0^\infty y^{\kappa + \delta -1 } | \p \{ X > y \} - \p \{ AX > y \} | \, \mathrm{d} y 
< 
\infty 
\end{equation}
for some $\delta > 0$,
where $A$ and $X$ are independent. Then there exists a function $q \in \mathcal{Q}$ such 
that 
\begin{equation} \label{eq:imp2-lim}
\lim_{n \to \infty} m(nh) \, x^\kappa \, e^{\kappa nh} \, \p \{ X > x e^{nh} \} = q(x), 
\quad
x \in C_{q}. 
\end{equation}
\end{theorem}

Since $m$ is regularly varying, $m(\log x)$ is slowly varying, and  
$m(\log x + nh) \sim m(nh)$ as $n \to \infty$. For  a
continuous nonzero function $q$ formula (\ref{eq:imp2-lim}) implies
\[
\p \{ X > x \} \sim \frac{q(x)}{x^\kappa m(\log x)} \quad \text{as } x \to 
\infty.
\]

As in Theorem \ref{th:imp} it is possible to give a stronger condition, similar to 
(\ref{eq:sum-cond}), which implies that (\ref{eq:imp2-lim}) holds for all $x >0$. 
However, in the corresponding key renewal theorem below (Lemma \ref{lemma:keyren}) 
besides summability a growth condition is also needed. Therefore the resulting stronger 
condition would be unnatural and it would not be clear how to check its validity neither 
for perpetuity equation (\ref{eq:perpetuity}) nor for maximum equation 
(\ref{eq:maxeq}).

The maximum and perpetuity results are the following.

\begin{theorem} \label{th:max2}
Assume (\ref{eq:A-ass2}) and $\E |B|^\nu < \infty$ for some $\nu > \kappa$. Let $X$ be 
the 
unique solution of (\ref{eq:maxeq}). Then 
there exists a function $q \in \mathcal{Q}$ such that
\[
\lim_{n \to \infty} m(nh) x^\kappa e^{\kappa nh}  \p \{ X > x e^{nh} \} = q(x),
\quad x \in C_q.
\]
If $B \geq 0$ a.s.~and $\p \{ B > 0 \} > 0$ then $q(x) > 0$.
\end{theorem}

In the special case $B \equiv 1$ this theorem was obtained by Korshunov \cite[Theorem 
2]{Kor}.

\begin{theorem} \label{th:perp2}
Assume (\ref{eq:A-ass2}) and $\E |B|^\nu < \infty$ for some $\nu > \kappa$. Let $X$ be 
the unique solution of (\ref{eq:perpetuity}).
Then there exist functions $q_1, q_2 \in \mathcal{Q}$ such that 
\begin{equation*} 
\begin{split}
& \lim_{n \to \infty} m(nh) x^\kappa e^{\kappa nh}  \p \{ X > x e^{ nh} \} = q_1(x),
\quad x \in C_{q_1},\\
& \lim_{n \to \infty} m(nh)  x^\kappa e^{\kappa nh}  \p \{ X < -x e^{ nh} \} = 
q_2(x), 
\quad x \in C_{q_2}.
\end{split}
\end{equation*}
If $\p \{ A x + B = x\} < 1$ for any $x \in \R$ then $q_1(x) + q_2 (x) > 0$.
\end{theorem}

Note that we only state convergence in continuity points in both cases.

The set of possible $q$ functions contains the nonzero right-continuous functions 
in $\mathcal{Q}$, i.e.~Proposition \ref{prop:Qset} is valid in this case too.

\subsection{Beyond Cram\'er's condition} \label{subsect:<1}

Assume now that $\E A^\kappa = \theta < 1$ for some $\kappa > 0$, and
$\E A^t = \infty$ for any $t > \kappa$. The assumption $\E A^t = \infty$ for all $t > 
\kappa$ means that $F_\kappa$ is heavy-tailed. The same renewal method leads now 
to a \emph{defective} renewal equation. To analyze the asymptotic behavior of the 
resulting equation we extend the results by Asmussen, Foss and Korshunov \cite[Section 
6]{AFK} to the arithmetic case.

Assume that $H$ is the distribution function of an arithmetic random variable with span 
$h$. 
Let $p_k = H(kh) - H(kh-)$, $k \in \Z$, and $p^{*2}_k = (H * H)(kh) - (H * H)(kh-)$. 
Then $H$ is $h$-subexponential, $H \in \mathcal{S}_h$, if 
$p_{n+1} \sim p_n$ and $p^{*2}_n \sim 2 p_n$ as $n \to \infty$. (According to the 
terminology introduced by 
Asmussen, Foss and Korshunov \cite{AFK} for distributions on $[0,\infty)$ and by Foss, 
Korshunov and Zachary \cite[Section 4.7]{FKZ} for distributions on $\R$, these 
distributions are $(0,h]$-subexponential.) In order to use a slight extension of Theorem 
5 
\cite{AFK} we need the additional natural assumption $\sup_{k \geq n} p_k =  O(p_n)$ as 
$n \to \infty$. 
Although in \cite{AFK} the distributions are concentrated on $(0,\infty)$ the results 
remain true in our setup due to the extra growth assumption. We refer to 
\cite[Appendix]{Kevei}.

Our assumptions on $A$ are the following:
\begin{equation} \label{eq:A-ass3}
\begin{gathered}
\E A^\kappa = \theta < 1, \  \kappa > 0, \
F_\kappa \in \mathcal{S}_{h}, \ \sup_{k \geq n} p_k =  O(p_n) 
\text{ as $n \to \infty$,}\\
\text{ and $\log A$ conditioned on $A \neq 0$ is arithmetic with span $h$}. 
\end{gathered}
\end{equation}

\begin{theorem} \label{th:imp3}
Assume (\ref{eq:A-ass3}) and (\ref{eq:delta-int}) for some $\delta > 0$.
Then there exists a function $q \in  \mathcal{Q}$ such that
\begin{equation} \label{eq:imp3-lim}
\lim_{n \to \infty} p_n^{-1} \,
x^\kappa \, e^{\kappa nh} \, \p \{ X > x e^{nh} \} = q(x), \quad x \in C_{q}. 
\end{equation}
with $p_n = F_\kappa(nh) - F_\kappa(nh-)$.
\end{theorem}

For a possible stronger version of (\ref{eq:imp3-lim}) which holds for all $x \in \R$ 
see the comment after Theorem \ref{th:imp2}.

The corresponding maximum and perpetuity results are the following.

\begin{theorem} \label{th:max3}
Assume (\ref{eq:A-ass3}) and $\E |B|^\nu < \infty$ for some $\nu > \kappa$. Let $X$ be 
the unique solution of (\ref{eq:maxeq}). Then 
there exists a function $q \in \mathcal{Q}$ such that
\[
\lim_{n \to \infty} p_n^{-1} \, x^\kappa \,  e^{\kappa nh} \,  \p \{ X > x e^{nh} \} = 
q(x),
\quad x \in C_q.
\]
If $B \geq 0$ a.s.~and $\p \{ B > 0 \} > 0$ then $q(x) > 0$.
\end{theorem}

\begin{theorem} \label{th:perp3}
Assume (\ref{eq:A-ass3}) and $\E |B|^\nu < \infty$ for some $\nu > \kappa$. Let $X$ be 
the unique solution of (\ref{eq:perpetuity}). 
Then there exist functions $q_1, q_2 \in \mathcal{Q}$ such that 
\begin{equation*} 
\begin{split}
& \lim_{n \to \infty} p_n^{-1} \, x^\kappa \, e^{\kappa nh} \, \p \{ X > x e^{ nh} \} = 
q_1(x),
\quad x \in C_{q_1},\\
& \lim_{n \to \infty} p_n^{-1} \, x^\kappa \, e^{\kappa nh} \,  \p \{ X < -x e^{ nh} \} = 
q_2(x), 
\quad x \in C_{q_2},
\end{split}
\end{equation*}
with $p_n = F_\kappa(nh) - F_\kappa(nh-)$. If $\p \{ A x + B = x\} < 1$ for any $x \in 
\R$, then $q_1(x) + q_2 (x) > 0$.
\end{theorem}

Proposition \ref{prop:Qset} remains true in this setup.

\subsection{Iterated function systems} \label{subsect:sandwich}

In this subsection we show that using Alsmeyer's sandwich method \cite{Alsmeyer} our 
results extend naturally to a more general framework. 

The Markov chain $(X_n)_{n \in \N}$ is an \emph{iterated function system of 
iid Lipschitz maps} (IFS) if $X_{n+1} = \Psi(\theta_{n+1}, X_n)$, $ n \in \N$, where
$\theta, \theta_1, \theta_2, \ldots$ are iid random vectors in $\R^d$, $d \geq 1$, the 
initial value $X_0$ is independent of the $\theta$'s, and $\Psi: \R^d \times \R \to \R$ 
is 
a measurable function, which is Lipschitz continuous in the second argument, 
i.e.~for all $\vartheta$ there exists $L_\vartheta > 0$ such that for all $x,y \in \R$
\[
| \Psi(\vartheta, x) - \Psi(\vartheta, y ) | \leq L_\vartheta |x - y|. 
\]
For theory and examples (and for a more general definition) we refer to 
Alsmeyer \cite{Alsmeyer}, Buraczewski, Damek and Mikosch \cite[Section 5]{BDM} and to 
Diaconis and Freedman \cite{DiaconisFreedman}. 


Under general conditions the stationary solution of the IFS exists and satisfies the 
random fixed point equation 
\begin{equation} \label{eq:RFPE}
X \stackrel{\mathcal{D}}{=} \Psi(\theta, X),
\end{equation}
where $\theta$ and $X$ on the right-hand side are independent. Therefore the 
corresponding implicit renewal theorem works and we obtain tail asymptotic for the 
solution $X$. The crucial difficulty here is the same as in the nonarithmetic case (see 
the remark after Theorem 2.3 \cite{Goldie}), namely to determine whether $q$ is nonzero 
or not. For equations (\ref{eq:perpetuity}) and (\ref{eq:maxeq}) there are reasonably 
good sufficient conditions for the strict positivity of the function $q$ (of the 
constant, in the arithmetic case). The main idea in \cite{Alsmeyer} is to 
find lower and upper bound for $\Psi$ such that
\[
A x \vee B = F(\theta, x) \leq \Psi(\theta, x) \leq G(\theta, x)  = A x + B'
\]
holds a.s.~with some (random) $A, B, B'$. Now, if $(A,B)$ and $(A,B')$ satisfies the 
conditions of Theorem \ref{th:max} and \ref{th:grinc}, respectively, then the tail of the 
solution $X$ of (\ref{eq:RFPE}) satisfies (\ref{eq:imp-lim}) with strictly positive $q$.
In particular, Theorems 5.3 and 5.4 in \cite{Alsmeyer} remain true in the arithmetic case.

Finally, we mention that there is no need to restrict ourselves to the finite mean case. 
Assuming (\ref{eq:A-ass2}) or (\ref{eq:A-ass3}) the corresponding version of Theorem 5.3 
and 5.4 in \cite{Alsmeyer} holds. The same results hold in the nonarithmetic case 
treated in \cite{Kevei}.

\section{Proofs} \label{sect:proofs}

First we prove Lemma 1 and Proposition \ref{prop:Qset}, since they are independent of the 
rest of the proofs.

\begin{proof}[Proof of Lemma \ref{lemma:cont-q}]
We show that every sequence $x_n \uparrow \infty$ contains a subsequence $x_{n_k}$ such 
that
\[
\lim_{k \to \infty} x_{n_k}^\kappa q^{-1}(x_{n_k}) \p \{ X > x_{n_k} \} = 1.
\]
This is equivalent to the statement.

Let us write $x_n = z_n e^{\ell_n h}$ with 
\[
z_n = \exp \left( h \left\{ \frac{\log x_n}{h} \right\} \right), \quad
\ell_n = \left\lfloor \frac{\log x_n}{h} \right\rfloor.
\]
Since $z_n \in [1,e^{h})$ by the Bolzano--Weierstrass theorem there is a 
subsequence $n_k$ such that $\lim_{k \to \infty} z_{n_k} = \lambda \in [1,e^{h}]$. 
To ease the notation we write $n$ for $n_k$. For any $\varepsilon > 0$ there is an 
$n_\varepsilon$ such that $| z_n - \lambda | \leq \varepsilon$ for $n \geq 
n_\varepsilon$. Therefore, using also (\ref{eq:imp-lim})
\[
\begin{split}
\limsup_{n \to \infty} x_n^{\kappa} \p \{ X > x_n \} &=
\limsup_{n \to \infty} z_n^{\kappa} e^{\kappa \ell_n h} 
\p \{ X > z_n e^{\ell_n h} \} \\
& \leq 
\limsup_{n \to \infty}
\left( \frac{\lambda+ \varepsilon}{\lambda - \varepsilon} \right)^\kappa
(\lambda - \varepsilon)^\kappa e^{\kappa \ell_n h} 
\p \{ X > (\lambda - \varepsilon) e^{\ell_n h}  \} \\
& = \left( \frac{\lambda+ \varepsilon}{\lambda - \varepsilon} \right)^\kappa
q(\lambda - \varepsilon).
\end{split}
\]
The same argument gives the corresponding lower bound. Since $\varepsilon > 0$ is 
arbitrary we obtain
\begin{equation} \label{eq:q-liminfsup}
q(\lambda + ) \leq \liminf_{n \to \infty} x_n^{\kappa} \p \{ X > x_n \}
\leq \limsup_{n \to \infty} x_n^{\kappa} \p \{ X > x_n \} \leq q(\lambda -).
\end{equation}
Now the continuity of $q$ implies the statement.

Note that (\ref{eq:q-liminfsup}) holds for general $q$. We did not use the continuity, 
only the logarithmic periodicity.
\end{proof}

\begin{proof}[Proof of Proposition \ref{prop:Qset}]
Motivated by the St.~Petersburg example we assume that $h = \log 2$ and $\kappa =1$. 
Moreover, we only prove the statement for the right tail. The general case follows easily 
from this.

Let $H$ be a distribution function, such that $H(1-)=0, H(2-) = 1$. Let the joint 
distribution of $(A,B)$ be the following:
\begin{equation} \label{eq:AB}
\p \{ A= 2^\ell, B= 0 \} = (1-2p) p^{\ell}, \ \ell = 0,1, \ldots, \ 
\p \{ A = 0, B \leq x \} = \frac{p}{1-p} H(x), \ p \in (0,1/2).
\end{equation}

It is easy to check that $(A,B)$ satisfies the conditions of Theorem \ref{th:grinc} 
with $\kappa =1$, $h = \log 2$. Let $(A,B), (A_1, B_1), \ldots$ iid random vectors 
with distribution given in 
(\ref{eq:AB}). Since $A B = 0$ a.s.~the solution of 
the perpetuity equation (\ref{eq:perpetuity}) can be written as
\begin{equation} \label{eq:AB0repr}
X = B_1 + A_1 B_2 + A_1 A_2 B_3 + \ldots =  A_1 A_2 \ldots A_{N-1} B_N,
\end{equation}
where $N = \min\{ i \, : \, A_i = 0 \}$ is a geometric random variable with parameter 
$\p \{ A = 0 \} = p/(1-p)$, i.e.
\[ 
\p \{ N = k \} = \frac{p}{1-p} \left( \frac{1-2p}{1-p} \right)^{k-1},
\quad k = 1,2,\ldots.
\]
From (\ref{eq:AB0repr}) we also see that the solution of (\ref{eq:perpetuity}) and of 
(\ref{eq:maxeq}) are the same.
Given that $N=k$ 
the variables $A_1, \ldots, A_{k-1}, B_k$ are independent, $A_1, \ldots, A_{k-1}$ have 
distribution $\p \{ A = 2^\ell | A \neq 0 \} = (1-p) p^{\ell}$, $\ell=0,1,2, \ldots$, and
$B_k$ has df $H$. To ease the notation we introduce the iid sequence $Y, Y_1, Y_2, 
\ldots$ independent of $(A_i,B_i)_{i \in \N}$, such that 
$\p \{ Y = \ell \} = (1-p) p^\ell$, $\ell=0,1,2, \ldots$, and put
$S_k = Y_1 + \ldots + Y_k$. Let $x > 1$ and write $x = 2^n z$ with 
$n = {\lfloor \log_2 x \rfloor}$, $z =2^{\{ \log_2 x \}}$. Since $B \in [1,2)$ we have 
that
\begin{equation} \label{eq:tail-calc}
\begin{split}
\p \{  X > x \} & = \p \{ A_1 A_2 \ldots A_{N-1} B_N > x \} \\
& = \sum_{k=1}^\infty \p \{ N = k \} \p \{ A_1 A_2 \ldots A_{k-1} B_k > x | N=k \} \\
& = \sum_{k=1}^\infty \p \{ N = k \} \left( 
\p \{ S_{k-1} \geq  n + 1\} + \p \{ S_{k-1} = n \} [1 - H(z)] \right) \\
& = \p \{ S_{N-1} \geq n+1\} + \p \{ S_{N-1} = n \} [1 - H(z)].
\end{split}
\end{equation}
We compute the probabilities $\p \{ S_{N-1} = n\}$. By the independence of $N$ and the 
$Y$'s, after some straightforward calculation one has for $s \in [0,1]$
\[
\E s^{S_{N-1}} = \frac{1}{2(1-p)} + \frac{1-2p}{2(1-p)} \sum_{k=1}^\infty \frac{s^k}{2^k}.
\]
That is 
\[
\p \{ S_{N-1} = k \} =
\begin{cases}
\frac{1}{2(1-p)}, & k=0,\\
\frac{1-2p}{2(1-p)} 2^{-k}, & k=1,2,\ldots.
\end{cases}
\]
Thus $\p \{ S_{N-1} \geq n+1 \} = \frac{1-2p}{2(1-p)} 2^{-n}$, and so continuing 
(\ref{eq:tail-calc}) we have
\begin{equation} \label{eq:tail-calc2}
\p \{ X > x \} = 
x^{-1} \,  \frac{1-2p}{2(1-p)} 2^{\{ \log_2 x \} } [ 2 - H(2^{\{ \log_2 x \}}) ].
\end{equation}
Let us choose now a right-continuous $q \in \mathcal{Q}$ (with the corresponding $\kappa$ 
and $h$) such that $q(2-) \in (0,1)$, otherwise $q$ is arbitrary. Let us choose $p, H$ in 
(\ref{eq:AB}) as 
\[
p = 1 - [2 - q(2-) ]^{-1}, \quad
H(y) = 
\begin{cases}
0, & y < 1, \\
2 - \frac{2(1-p)}{1 - 2 p} \frac{q(y)}{y}, & y \in [1,2), \\
1, & y \geq 1.
\end{cases}
\]
Since $q(y)/y$ is nonincreasing and right-continuous this is a distribution function. 
Substituting this back into (\ref{eq:tail-calc2}) we see that the tail is as stated.

To get rid of the condition $q(2-) \in (0,1)$ one only has to note that if $q(x)$ 
corresponds to $(A,B)$ then $c q(x/c)$ corresponds to $(A, cB)$, $c > 0$. Thus the proof 
is complete.
\end{proof}

In particular, with the choice 
\[
H(y) =
\begin{cases}
0, & y \leq 1, \\
2 - \frac{2}{y}, & y \in [1,2], \\
1, & y \geq 1,
\end{cases}
\]
in (\ref{eq:AB}) we obtain $\p \{ X > x \} = (2 - 1/(1-p)) x^{-1}$, $x > 2$, which is
regularly varying.

\begin{proof}[Proof of Theorem \ref{th:imp}]
We follow Grincevi\v{c}ius \cite[Theorem 2]{Grinc} and Goldie \cite[Theorem 2.3]{Goldie}. 

Introduce the notation
\[
\psi(x) = e^{\kappa x} [ \p \{ X > e^x \} - \p \{ AX > e^x \} ], \quad
f(x) = e^{\kappa x} \p \{ X > e^x\}.
\]
From the definition of $\psi$, using the independence of $X$ and $A$ we easily obtain the 
renewal equation
\begin{equation} \label{eq:renewal}
f(x) = \psi(x) + \E_\kappa f(x - \log A),
\end{equation}
where $\E_\kappa$ stands for the expectation under the measure $\p_\kappa$ defined 
in (\ref{eq:def-pkappa}). (See the proof of Theorem 3.2 in \cite{Goldie}, or the proof 
of 
Theorem 5 in \cite{Kevei}.) Introduce the smoothing of $g$ as
\begin{equation*} 
\widehat g(s) = \int_{-\infty}^s e^{-(s-x)} g(x) \mathrm{d} x. 
\end{equation*}
Applying this transform to both sides of (\ref{eq:renewal})
we get the renewal equation
\begin{equation} \label{eq:ren-smooth}
\widehat f(s) = \widehat \psi(s) + \E_\kappa \widehat f(s - \log A).
\end{equation}
For the solution we have (see again the proof of \cite[Theorem 5]{Kevei})
\begin{equation} \label{eq:smoothrenew}
\widehat f(s) = \int_\R \widehat \psi(s-y) U(\mathrm{d} y), 
\end{equation}
where $U(x) = \sum_{n=0}^\infty F_\kappa^{*n}(x)$ is the renewal function from 
(\ref{eq:def-U}). 

In order to apply the key renewal theorem (Theorem 6.2.6 in \cite{Iksanov}) we have to 
check that $\sum_{j \in \Z} | \widehat \psi(x + j h) | < \infty$ for any $x \in \R$. This 
follows from the direct Riemann integrability of $\widehat \psi$, which is proved in the 
course of the proof of \cite[Theorem 5]{Kevei}. For completeness and since we need the 
same calculation (without $| \cdot |$) we give a proof here. Using Fubini's theorem, 
after some calculation we have for any $x \in \R$
\[
\begin{split}
\sum_{j \in \Z} | \widehat \psi(x + j h) | 
& \leq \sum_{j \in \Z} \int_{-\infty}^\infty I (x + j h \geq y) e^{-(x+jh - y)} 
|\psi(y)| 
\mathrm{d} y \\ 
& = \int_{-\infty}^\infty \frac{1}{1 - e^{-h}} e^{-(x-y) - h \lceil (y-x) / h \rceil} | 
\psi(y) | \mathrm{d} y \\
& \leq \frac{1}{1 - e^{-h}}
\int_{-\infty}^\infty |\psi(y) | \mathrm{d} y < \infty.
\end{split}
\]
Therefore we may apply the 
key renewal theorem and we get
\begin{equation} \label{eq:lim-smooth}
\lim_{n \to \infty} \widehat f(s + nh) = C(s). 
\end{equation}
where, using the same calculation as above
\begin{equation} \label{eq:def-C}
\begin{split}
C(s) & = \frac{h}{\mu} \sum_{j \in \Z} \widehat \psi(s + j h) \\
& = \frac{h}{\mu} \frac{1}{1 - e^{-h}}
\int_{-\infty}^\infty  e^{-(s-y) - h \lceil (y-s) / h \rceil} \psi(y) \mathrm{d} y \\
&= \frac{h}{\mu} \frac{1}{1 - e^{-h}}
\int_{-\infty}^\infty  e^{-h \{ (s-y)/ h \} } \psi(y) \mathrm{d} y ,
\end{split}
\end{equation}
with $\mu = \E_\kappa \log A = \E A^{\kappa} \log A < \infty$.

We `unsmooth' (\ref{eq:lim-smooth}) the same way as in \cite{Grinc}. Using the definition 
of $\widehat f$, multiplying by $e^s$ we obtain from (\ref{eq:lim-smooth}) that for any 
$0 
< s_1 \leq s_2$
\[
\lim_{n \to \infty} e^{-nh} \int_{e^{s_1} e^{nh}}^{e^{s_2} e^{nh}} u^\kappa \p \{ X > u\} 
\mathrm{d} u = e^{s_2} C(s_2) - e^{s_1} C(s_1).
\]
Changing variables this reads
\begin{equation} \label{eq:q-det}
\lim_{n \to \infty} \int_{e^{s_1}}^{e^{s_2}}
(y e^{nh})^\kappa \p \{ X > y e^{nh} \} \mathrm{d} y = e^{s_2} C(s_2) - e^{s_1} C(s_1). 
\end{equation}
Since this holds for any $s_1 \leq s_2$ we readily obtain that the integrand remains 
bounded, therefore there exists a subsequence $n_k \uparrow \infty$ and a function $q$ 
such that $(y e^{n_k h})^\kappa \p \{ X > y e^{n_k h} \} \to q(y)$ for any $y \in C_q$. 
As a limit of nonincreasing functions $q(y) y^{-\kappa}$ is nonincreasing. Moreover, from 
(\ref{eq:q-det}) we see that
\begin{equation} \label{eq:q-int}
\int_{e^{s_1}}^{e^{s_2}} q(y) \mathrm{d} y = e^{s_2} C(s_2) - e^{s_1} C(s_1), 
\end{equation}
which determines $q$ uniquely at its continuity points. This implies that
$(y e^{n h})^\kappa \p \{ X > y e^{n h} \} \to q(y)$ holds true for the whole sequence of 
natural numbers whenever $y \in C_q$. From the latter we obtain the multiplicative 
periodicity $q(e^h y) = q(y)$. Since $y^{-\kappa} 
q(y)$ is nonincreasing $C_q$ is at most countable. Thus the first statement is 
completely proved.

Assume now that $\sum_{j \in \Z} | \psi(x + jh) | < \infty$ for any $x \in \R$. Then 
there is no need for the smoothing. Indeed, we may apply the key renewal theorem directly 
for the equation (\ref{eq:renewal}) and we obtain
\[
\lim_{n \to \infty} x^\kappa e^{\kappa nh} \p \{ X > x e^{nh} \} = q(x), 
\]
which is exactly the statement. The stated properties of $q$ follow easily. In fact
\begin{equation} \label{eq:q-form}
q(x) = \frac{h}{\mu} \sum_{j \in \Z}  \psi(\log x + jh).
\end{equation}
\end{proof}

\begin{remark} \label{rem:positivity}
Note that (\ref{eq:q-int}) implies $q(v) = (v C(\log v))'$ Lebesgue almost everywhere, 
from which, under some extra assumptions, some calculations give (\ref{eq:q-form}). 
Also note that $q(x) \equiv 0$ if and only if $v C(\log v)$ is constant. Since,
\[
v C( \log v) = \frac{h}{\mu (1 - e^{-h})} 
\int_{-\infty}^\infty e^{h \lfloor (\log v - y ) / h \rfloor} e^y \psi(y) \mathrm{d} y ,
\]
we see that if $\psi$ is nonnegative then $q(x) \equiv 0$ if and only if $\psi(y) \equiv 
0$. This readily implies the positivity of the function $q$  when $B \geq 0$ a.s.~and
$\p \{ B > 0 \} > 0$ in case of both the perpetuity equation (\ref{eq:perpetuity}) and  
the maximum equation (\ref{eq:maxeq}).
\end{remark}

\begin{proof}[Proof of Theorem \ref{th:grinc}]
We only have to show that $q_1(x) + q_2(x) > 0$. Goldie's argument \cite[p.~157]{Goldie} 
shows that it is enough to prove the positivity of the function for the maximum of the 
corresponding random walk. This was shown in \cite[Theorem 1.3.8]{Iksanov}.
\end{proof}

\begin{proof}[Proof of Theorem \ref{th:imp2}]
Recall the notations from the proof of Theorem \ref{th:imp}. Exactly the same way as in 
the previous proof we obtain the renewal equation (\ref{eq:ren-smooth}), which has a 
unique bounded solution (\ref{eq:smoothrenew}).
We want to apply the key renewal theorem in the infinite mean case. In order to do so, we 
first have to prove such a result.

The following simple lemma is the arithmetic analogue of \cite[Theorem 3]{Erickson}, 
\cite[Proposition 6.4.2]{Iksanov2}, \cite[Lemma 1]{Kevei}. We note that the statement 
holds under less restrictive condition on the left tail, see \cite[Proposition 
6.4.2]{Iksanov2}. However, for our purposes this weaker version is sufficient.

\begin{lemma} \label{lemma:keyren}
Assume (\ref{eq:SRT}) and (\ref{eq:A-ass2}). 
Let $z$ be a function such that $\sum_{j \in \Z} |z(x + jh)| < \infty$ for any $x \in \R$ 
and  $z(x) = O(x^{-1})$. Then
\[
\lim_{ n \to \infty} 
m(nh) \int_\R z(x+ nh -y) U(\mathrm{d} y) = h C_\alpha \sum_{j \in \Z} z(x + jh).
\]
\end{lemma}

\begin{proof}
We have
\[
\begin{split}
\int_\R z(x + nh -y) U(\mathrm{d} y) & 
= \sum_{j \in \Z} z(x + nh - jh) u_j = \sum_{k \in \Z } z(x + kh) u_{n-k} \\
& = \Big( \sum_{k \leq 0} + \sum_{1 \leq k \leq n} + \sum_{k > n} \Big) z(x + kh) u_{n-k}
= I_1 + I_2 + I_3.
\end{split}
\]
Recall that $m$ in (\ref{eq:def-m}) is regularly varying with parameter $1-\alpha$ and 
nondecreasing.
For $I_1$
\[
m(nh) I_1 = \sum_{k \leq 0}  z(x + kh) m((n-k)h) u_{n-k} \frac{m(nh)}{m((n-k)h)}
\to h C_\alpha  \sum_{k \leq 0}  z(x + kh),
\]
since the summands converge and $m(nh) / m((n-k)h) \leq 1$, thus Lebesgue's dominated 
convergence theorem applies. 
To handle $I_2$ let $1 > \delta > 0$ arbitrarily small. Then, from Potter bounds 
\cite[Theorem 1.5.6]{BGT} we obtain $\frac{m(nh)}{m((n-k)h)} \leq 2 \delta^{-1}$ for $n$ 
large enough and $k \leq (1-\delta) n$, thus by Lebesgue's dominated convergence theorem  
\[
\sum_{k =1}^{(1-\delta)n}  z(x + kh) m((n-k)h) u_{n-k} \frac{m(nh)}{m((n-k)h)}
\to h C_\alpha  \sum_{k \geq 1}  z(x + kh) \quad \text{as } n \to \infty. 
\]
Furthermore, noting that $U(y) \sim \sin (\pi \alpha) / (\pi \alpha) \ y^\alpha / 
\ell(y)$ as $y \to \infty$, for some $c> 0$ we have
\[
\sum_{k =(1-\delta)n}^n  |z(x + kh)| m(nh) u_{n-k} \leq \sup_{y > 0} y|z(y)|
\frac{m(nh)}{nh} U(\delta n h) \leq c \delta^\alpha. 
\]
Since $\delta > 0$ is arbitrarily small, we obtain
\[
\lim_{n \to \infty} m(nh) I_2   = h C_\alpha  \sum_{k \geq 1}  z(x + kh).
\]
Finally, for $I_3$
\[
m(nh) \sum_{k > n} |z(x + kh)| u_{n-k} \leq \sup_{y > 0} y |z(y)| \, U(0) 
\frac{m(nh)}{nh} 
\to 0.
\]
\end{proof}

In the proof of Theorem 5 \cite{Kevei} it is shown that under our conditions 
\begin{equation} \label{eq:psi-est}
\widehat \psi(s) = O(e^{-\delta s}) \quad \text{as } s \to \infty,
\end{equation}
therefore the condition of Lemma \ref{lemma:keyren} is satisfied, from which
\begin{equation} \label{eq:smooth-lim-inf}
\lim_{n \to \infty} m(nh) \widehat f(s + nh) = C(s) :=
h C_\alpha \sum_{j \in \Z} \widehat \psi(s + jh). 
\end{equation}
with the same $C$ as in (\ref{eq:def-C}).

Using the definition of $\widehat f$, multiplying by $e^s$ we obtain from 
(\ref{eq:smooth-lim-inf}) that for any $0 < s_1 \leq s_2$
\[
\lim_{n \to \infty} m(nh) e^{-nh} 
\int_{e^{s_1} e^{nh}}^{e^{s_2} e^{nh}} u^\kappa \p \{ X > u\} \mathrm{d} u = e^{s_2} 
C(s_2) - 
e^{s_1} C(s_1).
\]
Changing variables this reads
\begin{equation*} \label{eq:q-det2}
\lim_{n \to \infty} \int_{e^{s_1}}^{e^{s_2}} m(nh)
(y e^{nh})^\kappa \p \{ X > y e^{nh} \} \mathrm{d} y = e^{s_2} C(s_2) - e^{s_1} C(s_1). 
\end{equation*}
As in the previous proof this implies that
$ m(nh) (y e^{n h})^\kappa \p \{ X > y e^{n h} \} \to q(y)$ holds true for the whole 
sequence of natural numbers whenever $y \in C_q$ with some $q$, which satisfies the 
stated properties.
\end{proof}

\begin{proof}[Proof of Theorems \ref{th:max2} and \ref{th:perp2}]
We only have to prove that the assumptions imply the integrability condition in Theorem 
\ref{th:imp2}. This is done in the proof of Theorem 1 and 2 in \cite{Kevei}. 

Remark \ref{rem:positivity} implies $q(x) > 0$ in Theorem \ref{th:max2}.
Now, the strict positivity of  $q_1(x) + q_2(x)$ follows again from Goldie's 
argument \cite[p.157]{Goldie} and from the just proved positivity of $q$ in Theorem 
\ref{th:max2}.
\end{proof}

Before the proof of Theorem \ref{th:imp3} we need a key renewal theorem in the arithmetic 
case for defective distribution functions. The following statement is an extension to 
the arithmetic case of Theorem 5(i) \cite{AFK}. Recall $p_n$ from Theorem \ref{th:imp3}.

\begin{lemma} \label{lemma:keyren2}
Assume (\ref{eq:A-ass3}), $\sum_{j \in \Z} |z(x + jh) | < \infty$ for any $x \in \R$ and 
$\sup_{x \in [0,h]} z(x + nh) = o(p_n)$ as $n \to \infty$. Let $U(x) = \sum_{n=0}^\infty 
(\theta F_\kappa)^{*n}(x)$. Then
\[
\lim_{n \to \infty} p_n^{-1} \int_\R z(x + nh -y) U(\mathrm{d} y) = 
\frac{\theta  }{(1 - \theta)^2} \sum_{j \in \Z} z(x + jh).
\]
\end{lemma}

\begin{proof}
Note that Proposition 12 \cite{AFK} remains true in our case. Therefore
\begin{equation} \label{eq:u-asympt}
u_n = U(nh ) - U(nh-) \sim \frac{\theta}{(1 - \theta)^2} [F_\kappa(nh) - F_\kappa(nh-)]
= \frac{\theta}{(1 - \theta)^2} p_n.
\end{equation}
Since $\lim_{n \to \infty} p_n / p_{n+1} = 1$, there is sequence $\ell_n < n/2$ tending 
to infinity such that 
\[
\lim_{n \to \infty} \max_{|\ell| \leq \ell_n} u_n / u_{n+\ell} =1.
\]
Therefore
\[
\sum_{ |\ell| \leq \ell_n } z(x + \ell h ) u_{n-\ell} \sim u_n \sum_{\ell \in \Z} z(x + 
\ell h) \sim  \frac{\theta}{(1 - \theta)^2} p_n \sum_{\ell \in \Z} z(x + \ell h).
\]
Thus we only have to show that the remaining terms are $o(p_n)$.
For $\ell \leq - \ell_n$ using that $\max_{k \geq n } p_k = O(p_n)$ we obtain
\[
\sum_{ \ell \leq -\ell_n } z(x + \ell h ) u_{n-\ell}  = O(p_n) o(1).
\]
Using $z(x + n h) = o(p_n)$, (\ref{eq:u-asympt}) and Proposition 2 in \cite{AFK}
\[
\sum_{\ell = \ell_n}^{n-\ell_n} z(x + \ell h) u_{n-\ell}
= o(1) \sum_{\ell = \ell_n}^{n-\ell_n} p_\ell p_{n-\ell} = o(p_n).
\]
Finally, $z(x + n h) = o(p_n)$ and $\max_{k \geq n } p_k = O(p_n)$ imply
\[
\sum_{\ell> n-\ell_n} z(x + \ell h) u_{n-\ell} = o(p_n),
\]
and the proof is complete.
\end{proof}

\begin{proof}[Proof of Theorem \ref{th:imp3}]
Following the same steps as in the proof of Theorem \ref{th:imp} we obtain
\[
\widehat f(s) = \int_\R \widehat \psi(s-y) U(\mathrm{d} y), 
\]
with the defective renewal function $U(x) = \sum_{n=0}^\infty (\theta F_\kappa)^{*n}(x)$. 
Since $\theta < 1$ we have $U(\R) = (1 -\theta)^{-1} < \infty$. 

As in (\ref{eq:psi-est}) we have $\widehat \psi(x) = O(e^{-\delta x})$ for some 
$\delta > 0$. The subexponentiality of $F_\kappa$ implies that
$\sup_{x \in [0,h]} \widehat \psi(x + nh) = o(p_n)$. That is, the condition of Lemma 
\ref{lemma:keyren2} holds, and we obtain the asymptotic
\[
\lim_{n \to \infty} p_n^{-1} \widehat f(s + nh) =  
\frac{\theta}{(1 - \theta)^2} \sum_{j \in \Z} \widehat \psi(s + jh).
\]
The proof can be finished  in exactly the same way as in Theorem \ref{th:imp2}.
\end{proof}

\begin{proof}[Proof of Theorems \ref{th:max3} and \ref{th:perp3}]
Again, the integrability condition in Theorem \ref{th:imp3} follows from the proof of  
Theorem 3 and 4 in \cite{Kevei}. The positivity of the functions follow as before.
\end{proof}

\medskip

\noindent
\textbf{Acknowledgement.} I thank Alexander Iksanov for suggesting the problem, and for 
his comments, which greatly improved the paper. This research was funded by a postdoctoral 
fellowship of the Alexander von Humboldt Foundation.

{\small
\def\cprime{$'$}

}


\end{document}